\documentclass[a4paper,abstracton]{scrartcl}
\usepackage{fullpage}

\usepackage[T1]{fontenc}
\usepackage[utf8]{inputenc}
\usepackage[british]{babel}
\usepackage{csquotes}
\usepackage{amssymb}
\usepackage{amsmath}
\usepackage{amsthm}
\usepackage{hyperref}
\usepackage[plain]{fancyref}
\usepackage{xcolor}
\usepackage{microtype}

\usepackage{newpxtext,newpxmath}
\setkomafont{disposition}{\bfseries}

\usepackage[style=alphabetic,natbib=true,maxbibnames=99,maxcitenames=2]{biblatex}
\usepackage{xifthen}
\newcommand{\TODO}[1][]{\textbf{\color{red}\ifthenelse{\isempty{#1}}{[TODO]}{[TODO: #1]}}}

\hypersetup{colorlinks,
	linkcolor={blue!80!black},
	citecolor={blue!60!black},
	urlcolor={blue!80!black}
}

\DeclareMathOperator{\rank}{\mathrm{rank}}
\DeclareMathOperator{\fchar}{\mathrm{char}}
\DeclareMathOperator*{\End}{\mathrm{End}}
\DeclareMathOperator*{\Aut}{\mathrm{Aut}}
\DeclareMathOperator{\GL}{GL}
\DeclareMathOperator{\Br}{Br}

\DeclareMathOperator{\Jac}{Jac}
\DeclareMathOperator{\Hom}{Hom}
\DeclareMathOperator{\Spec}{Spec}
\newcommand{\Z}{\mathbb{Z}}
\newcommand{\Q}{\mathbb{Q}}

\newcommand{\F}{\mathbb{F}}

\makeatletter
\def\mkfancyprefix#1#2{%
  \@namedef{fancyref#1labelprefix}{#1}%
  \begingroup\def\x{\endgroup\frefformat{plain}}%
    \expandafter\x\csname fancyref#1labelprefix\endcsname
      {\MakeLowercase{#2}\fancyrefdefaultspacing##1}%
  \begingroup\def\x{\endgroup\Frefformat{plain}}%
    \expandafter\x\csname fancyref#1labelprefix\endcsname
      {#2\fancyrefdefaultspacing##1}%
}
\makeatother

\mkfancyprefix{def}{Definition}
\newtheorem{lemma}{Lemma}
\mkfancyprefix{lemma}{Lemma}
\newtheorem{corollary}{Corollary}
\mkfancyprefix{corollary}{Corollary}
\newtheorem{remark}{Remark}
\mkfancyprefix{rem}{Remark}
\newtheorem{theorem}{Theorem}
\mkfancyprefix{thm}{Theorem}
\newtheorem{example}{Example}

\mkfancyprefix{table}{Table}

\title{Genus 2 Curves in Small Characteristic}
\author{Lukas Zobernig}
\date{}

\begin{document}
\maketitle

\begin{abstract}
We study genus 2 curves over finite fields of small characteristic. The $p$-rank $f$ of a curve induces a stratification of the coarse moduli space $\mathcal{M}_2$ of genus 2 curves up to isomorphism. We are interested in the size of those strata for all $f \in \{0,1,2\}$. In characteristic 2 and 3, previous results show that the \emph{supersingular} $f=0$ stratum has size $q$. We show that for $q=3^r$, over $\F_q$ the \emph{non-ordinary} $f=1$ and \emph{ordinary} $f=2$ strata are of size $q(q-1)$ and $q^2(q-1)$, respectively. We give results found from computer calculations which suggest that these formulas hold for all $p \leq 7$ and break down for $p > 7$.
\end{abstract}

\section{Introduction}
\citet{maisner2007} studied supersingular genus 2 curves in characteristic 2. \citet{howe2008} showed that over an algebraically closed field $k$ characteristic 3, the supersingular locus $\mathcal{S}_2$ of the coarse moduli space $\mathcal{M}_2$ is in bijection with $k$ via certain invariants of genus 2 curves; we will discuss these invariants in \Fref{sec:moduli_and_invs}.

In \Fref{sec:theorems}, we prove a similar result for the non-ordinary stratum of $\mathcal{M}_2$ in characteristic 3, again in terms of certain \emph{absolute} invariants. From this we obtain a complete classification of isomorphism classes of supersingular, non-ordinary, and ordinary genus 2 curves defined over a finite field $\F_q$ with $q = 3^r$ in terms of absolute invariants, and their counts. In \Fref{sec:comp_results} we present and discuss further computational results for $p > 3$.

Let us briefly recall some background on the $p$-torsion structure of abelian varieties, moduli and fields of definition of genus 2 curves, and previous results for those curves over finite fields of characteristic 2 and 3.

\subsection{The $p$-rank and $a$-number of Abelian Varieties and Curves}
Fix a prime $p$, and an algebraically closed field $k$ of characteristic $p$ containing $\F_p$. Consider the \emph{additive group scheme} $\mathbb{G}_a = \Spec k[X]$ and the \emph{multiplicative group scheme} $\mathbb{G}_m = \Spec k[X,X^{-1}]$. The kernel of the relative Frobenius on these yields the finite group schemes $\alpha_p \cong \Spec k[X]/X^p$ and $\mu_p \cong \Spec k[X]/(X^p-1)$, respectively. The Cartier dual of $\alpha_p$ is itself, and the Cartier dual of $\mu_p$ is the constant group scheme $\Z/p\Z$.

Let $A/k$ be an abelian variety of genus $g$ and consider its $p$-torsion $A[p]$ as a group scheme. The \emph{$p$-rank} of $A$ is given by $f = \dim_{\F_p} \Hom(\mu_p, A[p])$. Similarly, the $a$-number of $A$ is given by $a = \dim_k \Hom(\alpha_p, A)$. Thus, geometrically, $A[p](k) \cong (\Z/p\Z)^f$. It holds that $0 \leq f \leq g$ and $1 \leq a+f \leq g$.

\begin{example}[\citet{pries2008}]
Let $A/k$ be of genus 2, i.e. an abelian surface, then we have the following possible types:
\begin{center}
\begin{tabular}{c|c|c|c|c}
$f$ & $a$ & $A[p]$ & Type & Codim. \\
\hline
$2$ & $0$ & $L^2$ & ordinary & 0 \\
$1$ & $1$ & $L \oplus I_{1,1}$ & non-ordinary & 1\\
$0$ & $1$ & $I_{2,1}$ & supersingular & 2 \\
$0$ & $2$ & $I_{1,1} \oplus I_{1,1}$ & superspecial & 3
\end{tabular}
\end{center}
Here $L = \Z/p\Z \oplus \mu_p$ is the $p$-torsion of an ordinary elliptic curve, whereas $I_{1,1}$ is the $p$-torsion of a supersingular elliptic curve. It is also the unique $BT_1$ group scheme of rank $p$ and fits into the following non-split exact sequence: $0 \to \alpha_p \to I_{1,1} \to \alpha_p \to 0$. Similarly, $I_{2,1}$ is the unique $BT_1$ group scheme of rank $p^2$ with $p$-rank 0 and $a$-number 1. Note that a superspecial abelian surface is in particular supersingular. The codimension of the associated strata in the full moduli space of abelian surfaces $\mathcal{A}_2$ is given as well.
\end{example}

We define the $p$-rank and the $a$-number of a genus g curve $C$ as the corresponding invariants of its Jacobian $\Jac(C)$ as an abelian variety.

Assume $\fchar(k) \neq 2$ and consider a genus $g$ hyperelliptic curve $C$ defined by an equation $y^2 = f(x)$ for $f(x) \in k[x]$ of degree $2g + 1$ or $2g + 2$. Let $c_i$ denote the coefficient of $x^i$ in the expansion of $f(x)^{(p-1)/2}$, and define for $\ell=0,\dots,g-1$ the $g \times g$ matrix $A_\ell$ with entries $(A_\ell)_{i,j} = (c_{ip-j})^{p^\ell}$.
Following \citet{yui1978} and heeding \citet{achterhowe2019} we call $A_0$ the Cartier-Manin matrix of the curve $C$ and define the matrix $M = A_{g-1} \cdots A_1 A_0$. Then we have the following lemma.
\begin{lemma}\label{lemma:prank_anumber}
Let $C$ be a genus g hyperelliptic curve defined by $y^2 = f(x)$.
\begin{enumerate}
\item The $p$-rank of $C$ is $f_C = \rank(M)$.
\item The $a$-number of $C$ is $a_C = g - \rank(A_0)$.
\end{enumerate}
\end{lemma}

\subsection{Moduli Space and Invariants}\label{sec:moduli_and_invs}
Points in the coarse moduli space $\mathcal{M}_2$ of genus 2 curves up to isomorphism have various different descriptions. Consider the weighted projective space $S_k = \mathbb{P}(2,4,6,8,10)(k)$ over an arbitrary field $k$. \citet{igusa60} gave a set of invariants $[J_2:J_4:J_6:J_8:J_{10}] \in S_k$ associated to every genus 2 curve $C$ defined over $k$. Two curves are isomorphic over $k^\text{alg}$ if any only if their Igusa-invariants are the same in $S_k$. If $\fchar(k) \neq 2$, we can assume that $C$ is given by a model of the form $y^2 = f(x)$, where $f(x) \in k[x]$ is a polynomial of degree 5 or degree 6. Then $2^{12}J_{10}$ is simply the discriminant of $f(x)$.

For fields of characteristic different from 2, Cardona, Quer, Nart, and Pujol\`as \cite{cardonaquer2005, cardonanartpujolas2005} gave the absolute ``G2'' invariants 
\begin{equation}\label{eq:g2invariants}
(g_1,g_2,g_3) = \begin{cases}
\left(\frac{J_2^5}{J_{10}}, \frac{J_2^3 J_4}{J_{10}}, \frac{J_2^2J_6}{J_{10}}\right) &\text{if}\; J_2 \neq 0, \\
\left(0, \frac{J_4^5}{J_{10}^2}, \frac{J_4J_6}{J_{10}}\right) &\text{if}\; J_2 = 0, J_4 \neq 0, \\
\left(0, 0, \frac{J_6^5}{J_{10}^3}\right) &\text{if}\; J_2 = J_4 = 0.
\end{cases}
\end{equation}
Again, two curves are isomorphic over $k^\text{alg}$ if any only if their absolute invariants are the same. Hence, the $k$-points $\mathcal{M}_2(k)$ are in bijection with the set of points $(g_1,g_2,g_3) \in \mathbb{A}^3(k)$ as above.

\begin{remark}\label{rem:field_of_mod_def}
Consider a point of moduli $P = [C] \in \mathcal{M}_2(k)$ defined over $k$. Then \citet{mestre1991}, and \citet[Theorem 2]{cardonaquer2005} showed that it is only the generic case $\Aut(C) \cong C_2$ where there exists an obstruction to the point $P$ being represented by a curve $C'$ defined over $k$, which is an element in $\Br_2(k)$ (i.e. a non-trivial 2-torsion element). If the automorphism group $\Aut(C) \not\cong C_2$, then there always exists a curve $C'$ defined over $k$ with $C \cong C'$ (geometrically). Hence, if $k$ has trivial Brauer group (for example when $k$ is a finite field, or an algebraic closure thereof), then for every $(g_1,g_2,g_3) \in \mathbb{A}^3(k)$ there exists a curve $C$ defined over $k$ with those invariants.
\end{remark}

\subsection{Transformation of Sextic Forms}\label{sec:sextic_forms_transf}
In most generality, a genus 2 curve is given by an equation of the form $y^2 = f(x)$ with $f(x) = c_6x^6 + \cdots c_1 x + c_0$ for coefficients $c_0,\dots,c_6 \in k$. After homogenising, we find $F(X,Z) = c_6X^6 + \cdots c_1 XZ^5 + c_0Z^6$. Then the group $\GL_2$ acts on $F(X,Z)$ via the following transformation: Let $M = \begin{pmatrix} \alpha & \beta \\ \gamma & \delta \end{pmatrix} \in GL_2$ and let
\begin{align*}
X = \alpha X' + \beta Z', \\
Z = \gamma X' + \delta Z'.
\end{align*}
We find $F(X,Z) = c_6'X'^6 + \cdots c_1' X'Z'^5 + c_0'Z'^6 = G(X',Z')$ for new coefficients $c_0',\dots,c_6'$. Hence, we may call $F(X,Z)$ and $G(X',Z')$ to be $\GL_2$-conjugates. Let $g(x') = c_6'x'^6 + \cdots c_1' x' + c_0'$ be the dehomogenisation of $G(X',Z')$. Then the two curves $C: y^2 = f(x)$ and $C': y'^2 = g(x')$ are isomorphic via
\begin{equation*}
(x',y') \mapsto (x,y) = \left(\frac{\alpha x' + \beta}{\gamma x' + \delta}, \frac{y'}{(\gamma x' + \delta)^3}\right).
\end{equation*}
This is actually an if and only if relation: Two genus 2 curves $C,C'$ are isomorphic if any only if their associated sextic forms are conjugate under $\GL_2$-action. Thus we can also study points in the moduli space $\mathcal{M}_2$ by considering sextic forms up to $\GL_2$-conjugacy, see \citet{mestre1991}.

\subsection{Genus 2 Curves in Characteristic 2 and 3}\label{sec:prev_results}
\citet{maisner2007}, and \citet{howe2008} studied supersingular genus 2 curves in characteristic 2 and 3, respectively. \cite[Theorem 2.2]{howe2008} states that in characteristic 3, the coarse moduli space $\mathcal{S}_2$ of supersingular genus 2 curves is isomorphic to the affine line $\mathbb{A}^1$: The absolute invariants $(0,0,g_3)$ correspond to the curve $y^2 = x^6 + g_3^2x^3 + g_3^3 x + g_3^4$ if $g_3 \neq 0$ and the point $(0,0,0)$ corresponds to $y^2 = x^5 + 1$. Hence, over the finite field $\F_q$ where $q = 3^r$, the subspace $\mathcal{S}_2(\F_q)$ of $\mathcal{M}_2(\F_q)$ corresponding to isomorphism classes of supersingular genus 2 curves contains $q$ elements.

\section{Theoretical Results}\label{sec:theorems}
The goal of this section is to show that over the finite field $\F_q$ where $q = 3^r$, there are $q^2(q-1)$ many ordinary, $q(q-1)$ many non-ordinary, and $q$ many supersingular genus 2 curves.

\begin{theorem}\label{thm:nonord_classification}
Let $k$ be a finite field of characteristic $3$.
Let $C$ be a genus 2 curve defined over $k$ and let $(g_1,g_2,g_3)$ be its absolute invariants. Then $C$ is non-ordinary if and only if $g_1 = 0$, $g_2 \in k^\times$, and $g_3 \in k$.
\end{theorem}
\begin{proof}
We can assume that $C$ has a model over $k$ of the form $dy^2 = x^6 + c_5x^5 + c_4x^4 + c_3x^3 + c_2x^2 + c_1x + c_0$. Then the matrices $A_0$ and $A_1$ of $C$ are given by
\begin{equation*}
A_0 = \begin{pmatrix} c_2 & c_1 \\ c_5 & c_4 \end{pmatrix},\;
A_1 = \begin{pmatrix} c_2^3 & c_1^3 \\ c_5^3 & c_4^3 \end{pmatrix}.
\end{equation*}
Thus, the Hasse-Witt matrix $M$ of $C$ becomes
\begin{equation*}
M = A_1 A_0 = \begin{pmatrix} c_1^3 c_5 + c_2^4 & c_1^3 c_4 + c_1 c_2^3 \\ c_2 c_5^3 + c_4^3 c_5 & c_1 c_5^3 + c_4^4 \end{pmatrix}.
\end{equation*}

By \Fref{lemma:prank_anumber}, C has $p$-rank 1 if $\rank(M) = 1$. To make analysis easier, we can alternatively require $\rank(A_0) = 1$, and exclude the supersingular cases $dy^2 = x^6 + c_5x^5 + c_3x^3 + c_0$ and $dy^2 = x^6 + c_3x^3 + c_1x + c_0$ we find from the constraint $M = 0$ (again by \Fref{lemma:prank_anumber}). The remaining cases correspond to the following seven possible $A_0$ matrices:
\begin{equation*}
\begin{pmatrix} c_2 & 0\\ 0 & 0  \end{pmatrix},
\begin{pmatrix} 0 & 0\\ 0 & c_4  \end{pmatrix},
\begin{pmatrix} c_2 & 0\\ c_5 & 0  \end{pmatrix},
\begin{pmatrix} 0 & c_1\\ 0 & c_4  \end{pmatrix},
\begin{pmatrix} c_2 & 0\\ c_2 & 0  \end{pmatrix},
\begin{pmatrix} 0 & c_5\\ 0 & c_5  \end{pmatrix},
\begin{pmatrix} c_2 & c_5\\ c_2 & c_5  \end{pmatrix}.
\end{equation*}

Both approaches yield the following seven possible types of models:
\begin{align*}
dy^2 &= x^6 + c_3x^3 + c_2x^2 + c_0, \\
dy^2 &= x^6 + c_4x^4 + c_3x^3 + c_0, \\
dy^2 &= x^6 + c_5x^5 + c_3x^3 + c_2x^2 + c_0, \\
dy^2 &= x^6 + c_4x^4 + c_3x^3 + c_1x + c_0, \\
dy^2 &= x^6 + c_3x^3 + c_2(x^2 + x) + c_0, \\
dy^2 &= x^6 + c_5(x^5 + x^4) + c_3x^3 + c_0, \\
dy^2 &= x^6 + c_5(x^5 + x^4) + c_3x^3 + c_2(x^2 + x) + c_0.
\end{align*}

These types are all pairwise isomorphic over $k^\text{alg}$ either by simply switching affine patches via the transformation $x \mapsto 1/x$ followed by a rescaling of $y$, or by a more general $\GL_2(k^\text{alg})$ transformation of the associated sextic forms, see \Fref{sec:sextic_forms_transf}.

Hence, without loss of generality, assume that $C$ has a model $dy^2 = x^6 + c_3x^3 + c_2x^2 + c_0$. For its Igusa invariants we find $\left[ 0 : c_2^3 : -c_0^3 - c_0c_2^3 + c_3^6 : -c_2^6 : -c_0c_2^6 \right]$. Using \Fref{eq:g2invariants}, the absolute invariants are
\begin{equation*}
\left(0, \frac{c_2^3}{c_0^2}, \frac{c_0^3 + c_0c_2^3 - c_3^6}{c_0c_2^3} \right) = (0,g_2,g_3).
\end{equation*}
Since the discriminant $2^{12}J_{10} = -c_0c_2^6$ has to be non-zero for the model to define a smooth curve, we have that $c_0,c_2 \neq 0$ and so $g_2 \in k^\times, g_3 \in k$.

On the other hand, choose $g_2 \in k^\times, g_3 \in k$, and set $c_2 = \sqrt[3]{g_2}$, $c_3 = \sqrt[6]{1 + g_2 - g_2g_3}$. Then by \Fref{rem:field_of_mod_def} the curve $y^2 = x^6 + c_3x^3 + c_2x^2 + 1$ is geometrically isomorphic to a curve defined over $k$ with absolute invariants $(0,g_2,g_3)$, and is non-ordinary by \Fref{lemma:prank_anumber}.
\end{proof}

The proof of \Fref{thm:nonord_classification} immediately implies the following corollary.
\begin{corollary}
Let $k$ be a finite field of characteristic 3. For $g_2 \in k^\times$ and $g_3 \in k$, the curve
\begin{equation*}
y^2 = x^6 + \sqrt[6]{1 + g_2 - g_2g_3}x^3 + \sqrt[3]{g_2}x^2 + 1
\end{equation*}
defined over $k^\text{alg}$ is non-ordinary and has absolute invariants $(0,g_2,g_3)$.
\end{corollary}

Combining \Fref{sec:prev_results} and \Fref{thm:nonord_classification} yields the following classification for the absolute invariants of an ordinary genus 2 curve over a finite field of characteristic 3.
\begin{corollary}
Let $k$ be a finite field of characteristic $3$. Every ordinary genus 2 curve defined over $k$ has absolute invariants $(g_1,g_2,g_3)$ with $g_1 \in k^\times$ and $g_2,g_3 \in k$.
\end{corollary}

For a finite field $k = \F_q$ with $q = p^r$, denote by $\mathcal{S}_2(k)$, $\mathcal{N}_2(k)$, and $\mathcal{O}_2(k)$ the supersingular, non-ordinary, and ordinary strata of the coarse moduli space $\mathcal{M}_2(k)$. Finally, we can use \cite[Theorem 2.2]{howe2008} and \Fref{thm:nonord_classification} to describe the sizes of these $p$-rank strata defined over $\F_q$ where $q = 3^r$.
\begin{theorem}\label{thm:strata_sizes}
Let $k$ be a finite field of characteristic 3 with $q$ elements. Then
\begin{enumerate}
\item $\#\mathcal{O}_2(k) = q^2(q-1)$,
\item $\#\mathcal{N}_2(k) = q(q-1)$,
\item $\#\mathcal{S}_2(k) = q$.
\end{enumerate}
\end{theorem}
\begin{proof}
The discussion in \Fref{sec:moduli_and_invs} implies that there are $\#\mathbb{A}^3(k) = q^3$ many geometric isomorphism classes of curves defined over $k$. Then the sizes of $\mathcal{O}_2(k)$, $\mathcal{N}_2(k)$, and $\mathcal{S}_2(k)$ immediately follow from \cite[Theorem 2.2]{howe2008} and \Fref{thm:nonord_classification}.
\end{proof}

\section{Computational Results for Small Primes}\label{sec:comp_results}
We have computed the $p$-rank of a curve corresponding to each absolute invariant $(g_1,g_2,g_3) \in \mathbb{A}^3(\F_q)$ for small finite fields $\F_q$ of characteristic $2,3,5$, and 7. \Fref{table:strata_sizes_exact} summarises the results we have found.

\begin{table}
\centering
\begin{tabular}{c|c|c|c}
q & $\#\mathcal{S}_2(k)$ & $\#\mathcal{N}_2(k)$ & $\#\mathcal{O}_2(k)$ \\
\hline
2 & 2 & 2 & 4 \\
$2^2$ & $2^2$ & $2^2 \cdot 3$ & $2^4 \cdot 3$ \\
3 & 3 & $3 \cdot 2$ & $3^2 \cdot 2$ \\
$3^2$ & $3^2$ & $3^2 \cdot 8$ & $3^4 \cdot 8$ \\
$3^3$ &  $3^3$ & $3^3 \cdot 26$ & $3^6 \cdot 26$ \\
5 & 5 & $5 \cdot 4$ & $5^2 \cdot 4$ \\
$5^2$ & $5^2$ & $5^2 \cdot 24$ & $5^4 \cdot 24$ \\
7 & 7 & $7 \cdot 6$ & $7^2 \cdot 6$ \\
$7^2$ & $7^2$ & $7^2 \cdot 48$ & $7^4 \cdot 48$
\end{tabular}
\caption{Strata sizes of $\mathcal{M}_2$ over finite fields of characteristic $2,3,5$, and $7$.}\label{table:strata_sizes_exact}
\end{table}
As expected, in characteristic $2$, $\#\mathcal{S}_2$ follows the results of \citet{maisner2007}. Similary, characteristic $3$ behaves as predicted by \Fref{thm:strata_sizes}. Inspired by the computational results, \emph{we conjecture that this behaviour also holds in characteristic 5 and 7 and for the other strata in characteristic 2}.

Denote by $\Delta_0 = \#\mathcal{S}_2(k) - q$, $\Delta_1 = \#\mathcal{N}_2(k) - q(q-1)$, and $\Delta_2 = \#\mathcal{O}_2(k) - q^2(q-1)$ the difference in number of points to what one would expect if the size of the strata followed the small prime behaviour. We have computed \Fref{table:strata_sizes_small_char} determining the $p$-rank of a curve corresponding to each absolute invariant $(g_1,g_2,g_3) \in \mathbb{A}^3(\F_q)$ for finite fields $\F_q$ with $q > 7$.

\begin{table}
\centering
\begin{tabular}{c|c|c|c|c|c|c}
q & $\#\mathcal{S}_2(k)$ & $\Delta_0$ & $\#\mathcal{N}_2(k)$ & $\Delta_1$ & $\#\mathcal{O}_2(k)$ & $\Delta_2$ \\
\hline
11 & 9 & -2 & 101 & -9 & 1221 & 11 \\
$11^2$ & 117 & -4 & 14403 & -117 & 1757041 & 121 \\
13 & 20 & 7 & 149 & -7 & 2028 & 0 \\
$13^2$ & 330 & 161 & 28231 & -161 & 4798248 & 0 \\
17 & 25 & 8 & 264 & -8 & 4624 & 0 \\
19 & 26 & 7 & 335 & -7 & 6498 & 0 \\
23 & 36 & 13 & 494 & -12 & 11637 & -1 \\
29 & 49 & 20 & 851 & 39 & 23489 & -59 \\
31 & 54 & 23 & 970 & 40 & 28767 & -63 \\
37 & 102 & 65 & 1229 & -103 & 49322 & 38 \\
41 & 70 & 29 & 1794 & 154 & 67057 & -183 \\
47 & 109 & 62 & 2308 & 146 & 101406 & -208 \\
53 & 155 & 102 & 2843 & 87 & 145879 & -189 \\
61 & 186 & 125 & 3775 & 115 & 223020 & -240 \\
67 & 210 & 143 & 4093 & -329 & 296460 & 186 \\
71 & 146 & 75 & 5770 & 800 & 351995 & -875 \\
73 & 269 & 196 & 4949 & -307 & 383799 & 111 \\
79 & 216 & 137 & 6838 & 676 & 485985 & -813 \\
83 & 259 & 176 & 7529 & 723 & 563999 & -899 \\
89 & 226 & 137 & 9053 & 1221 & 695690 & -1358 \\
97 & 408 & 311 & 8726 & -586 & 903539 & 275 \\
101 & 347 & 246 & 11784 & 1684 & 1018170 & -1930 \\
103 & 443 & 340 & 11394 & 888 & 1080890 & -1228 \\
107 & 357 & 250 & 12443 & 1101 & 1212243 & -1351 \\
109 & 412 & 303 & 11750 & -22 & 1282867 & -281 \\
113 & 417 & 304 & 12834 & 178 & 1429646 & -482 \\
127 & 570 & 443 & 17592 & 1590 & 2030221 & -2033 \\
131 & 409 & 278 & 20931 & 3901 & 2226751 & -4179 \\
137 & 576 & 439 & 18198 & -434 & 2552579 & -5 \\
139 & 516 & 377 & 20745 & 1563 & 2664358 & -1940
\end{tabular}
\caption{Strata sizes and differences of $\mathcal{M}_2$ over finite fields of small characteristic.}\label{table:strata_sizes_small_char}
\end{table}

It would be interesting to determine how exactly the differences $\Delta_1$ and $\Delta_2$ of the non-ordinary and ordinary strata, respectively, depend on $q$. Alternatively we would like to say something about their distribution depending on the characteristic $p$. This seems to require some Sato-Tate style argument depending on the reduction of the moduli space $\mathcal{M}_2$ at various primes.

\subsection{The Supersingular Locus}
In this section we recall some theory about the size of the supersingular locus $\mathcal{S}_2$. \citet{ibukiyamakatsuraoort1986}, \citet{katsuraoort1987}, and \citet{koblitz1975} show that the supersingular locus of $\mathcal{M}_2$ is a union of projective lines with the singular points of this union corresponding to the superspecial points. Denote the finite set of superspecial points by $\mathcal{SP}_2$.

Let $B$ be the definite quaternion algebra over $\Q$ with discriminant $D$ and let $\mathcal{O}$ be a maximal order of $B$. Write $D = D_1 D_2$ for two positive integers $D_1, D_2$ and set $L_n(D_1,D_2)$ the set of left $\mathcal{O}$-lattices in $B^n$ which are equivalent to $(\mathcal{O} \otimes \Z_p)^n$ at $p$ if $p$ does not divide $D_2$, and otherwise to the other local equivalence class if $p$ divides $D_2$, see \citet{ibukiyamakatsuraoort1986}. Denote by $H_n(D_1,D_2)$ the number of global equivalence classes in $L_n(D_1,D_2)$. In particular, for a prime number $p$, denote by $H_n(p,1)$ the class number of the \emph{principal genus} and by $H_n(1,p)$ the class number of the \emph{non-principal genus}.

Let $k$ be a finite field of characteristic $p$. On the one hand, it is known that every principally polarised superspecial abelian surface $A$ over $k^\text{alg}$ arises from choosing a principal polarisation on a product $E \times E$ for a supersingular elliptic curve $E$. Let $B = \End(E) \otimes \Q$ the endomorphism algebra of $E$, which is the definite quaternion algebra ramified at $p$. Then the number of principally polarisations on $E \times E$ over $k^\text{alg}$ up to automorphisms is equal to $H_2(p,1)$ for $B$.

On the other hand, as a principally polarised abelian variety, either $A = E_1 \times E_2$ for two supersingular elliptic curves $E_1$ and $E_2$, or $A = \Jac(C)$ for a superspecial genus 2 curve $C$. Hence the number of superspecial genus 2 curves over $k^\text{alg}$ is given by $H_p = H_2(p,1) - h_p(h_p+1)/2$, where $h_p = H_1(p,1)$ is the number of isomorphism classes of supersingular elliptic curves over $k^\text{alg}$, see \citet[Corollary 2.12]{ibukiyamakatsuraoort1986}.

The number $H_p$ is finite; every superspecial genus 2 curve can be defined over $\F_{p^2}$. This allows us to determine the finite contribution of superspecial genus 2 curves to $\#\mathcal{S}_2(k)$ when $k$ contains $\F_{p^2}$. For small primes $p$ we find \Fref{table:ssing_sspec_sizes}; we have also determined the number of superspecial curves defined over $\F_p$.

\begin{table}
\centering
\begin{tabular}{c|c|c|c|c|c|c|c|c|c}
$p$ & 11 & 13 & 17 & 19 & 23 & 29 & 31 & 37 & 41 \\
\hline
$\#\mathcal{S}_2(\F_p)$ & 9 & 20 & 25 & 26 & 36 & 49 & 54 & 102 & 70 \\
$H_p$ & 2 & 3 & 5 & 7 & 10 & 18 & 20 & 31 & 40 \\
$\#\mathcal{SP}_2(\F_p)$ & 2 & 3 & 5 & 5 & 8 & 12 & 12 & 9 & 22
\end{tabular}
\caption{Number of $\F_p$ points of $\mathcal{S}_2$ and $\mathcal{SP}_2$ in small characteristic.}\label{table:ssing_sspec_sizes}
\end{table}

The number of irreducible components of the supersingular locus of $\mathcal{M}_2$, i.e. the aforementioned projective lines, is equal to the class number of the non-principal genus $H_2(1,p)$, see \citet[Theorem 5.7]{katsuraoort1987}.

\section*{Acknowledgements}
We thank Yan Bo Ti for helpful discussions and Everett Howe for helpful comments.

\printbibliography

\end{document}